\documentclass{amsart}
\usepackage{amsfonts, amsbsy, amsmath, amssymb}

\usepackage{xcolor}

\ifx\fiverm\undefined 
  \newfont\fiverm{cmr5} 
\fi

\newtheorem{thm}{Theorem}[section]

\newtheorem{exmp}[thm]{Example}

\newtheorem{thm-con}[thm]{Theorem-Conjecture}
\numberwithin{equation}{section}

\theoremstyle{definition}

\allowdisplaybreaks

\newcommand{\f}{\Bbb F}
\newcommand{\X}{\boldsymbol X}
\newcommand{\x}{\boldsymbol x}

\begin{document}

\title[Hasse-Weil Bound and Rational Functions]{An Application of the Hasse-Weil Bound to Rational Functions over Finite Fields}

\author[Xiang-dong Hou]{Xiang-dong Hou}
\address{Department of Mathematics and Statistics,
University of South Florida, Tampa, FL 33620}
\email{xhou@usf.edu}

\author[Annamaria Iezzi]{Annamaria Iezzi}
\address{Department of Mathematics and Statistics,
University of South Florida, Tampa, FL 33620}
\email{aiezzi@usf.edu}




\begin{abstract}
We use the Aubry-Perret bound for singular curves, a generalization of the Hasse-Weil bound, to prove the following curious result about rational functions over finite fields: Let $f(X),g(X)\in\f_q(X)\setminus\{0\}$ be such that $q$ is sufficiently large relative to $\deg f$ and $\deg g$, $f(\f_q)\subset g(\f_q\cup\{\infty\})$, and for ``most'' $a\in\f_q\cup\{\infty\}$, $|\{x\in \f_q:g(x)=g(a)\}|>(\deg g)/2$. Then there exists $h(X)\in\f_q(X)$ such that $f(X)=g(h(X))$. A generalization to multivariate rational functions is also included.

\end{abstract}

\maketitle

\section{Introduction}
Let $\f_q$ denote the finite field with $q$ elements and $\overline \f_q$ be its algebraic closure. For a nonzero rational function $f(X)\in\f_q(X)$ written in the form $f(X)=A(X)/B(X)$, where $A,B\in\f_q[X]$ are such that $\text{gcd}(A,B)=1$, we define $\deg f=\max(\deg A,\deg B)$; when $\deg f>0$, we have $\deg f=[\f_q(X):\f_q(f(X))]$. We write $\f_q^\dagger=\f_q\cup\{\infty\}$, and for $a\in\f_q^\dagger$ and $f\in\f_q(X)$, we define $f|_{\f_q}^{-1}(a)=\{x\in\f_q:f(x)=a\}$.

A polynomial $F\in\f_q[X_1,\dots,X_n]$ is called {\em absolutely irreducible} if it is irreducible in $\overline \f_q[X_1,\dots,X_n]$. 
Let $\Bbb P^2(\f_q)$ denote the projective plane over $\f_q$. For a homogeneous polynomial $F(X,Y,Z)\in\f_q[X,Y,Z]$, define
\[
V_{\Bbb P^2(\f_q)}(F)=\{(x:y:z)\in \Bbb P^2(\f_q):F(x,y,z)=0\}.
\]
Assume that $F(X,Y,Z)\in\f_q[X,Y,Z]$ is an absolutely irreducible homogeneous polynomial of degree $d>0$. When the plane curve $V_{\Bbb P^2(\f_q)}(F)$ is smooth, the Hasse-Weil bound \cite{Stichtenoth-1993, Weil-1948} states that
\begin{equation}\label{1.1}
\bigl||V_{\Bbb P^2(\f_q)}(F)|-(q+1)\bigr|\le 2gq^{1/2},
\end{equation}
where $g$ is the genus of the curve $V_{\Bbb P^2(\f_q)}(F)$. Still assuming the absolute irreducibility of $F(X,Y,Z)$ but without assuming the smoothness of $V_{\Bbb P^2(\f_q)}(F)$, Aubry and Perret \cite{Aubry-Perret} proved that
\begin{equation}\label{AP}
\bigl||V_{\Bbb P^2(\f_q)}(F)|-(q+1)\bigr|\le (d-1)(d-2)q^{1/2}.
\end{equation}

For $F(X,Y)\in\f_q[X,Y]$, let $V_{\f_q^2}(F)=\{(x,y)\in\f_q^2:F(x,y)=0\}$. If $F(X,Y)\in\f_q[X,Y]$ is absolutely irreducible of degree $d>0$, then applying the Aubry-Perret bound to the homogenization of $F(X,Y)$ gives
\begin{equation}\label{1.1.1}
q+1-(d-1)(d-2)q^{1/2}-d\le|V_{\f_q^2}(F)|\le q+1+(d-1)(d-2)q^{1/2}.
\end{equation}
Loosely speaking, if $F(X,Y)\in\f_q[X,Y]$ is absolutely irreducible, then $|V_{\f_q^2}(F)|=q+O(q^{1/2})$ as $q\to\infty$. In contrast, if $F(X,Y)\in\f_q[X,Y]$ is irreducible but not absolutely irreducible, then we have
\begin{equation}\label{1.3}
|V_{\f_q^2}(F)|\le \frac 14(\deg F)^2.
\end{equation}
To see \eqref{1.3}, let $G(X,Y)$ be an irreducible factor of $F$ in $\overline\f_q[X,Y]$ and let $\f_{q^n}$ be the smallest extension of $\f_q$ such that $G(X,Y)\in\f_{q^n}[X,Y]$. Then $F=a\prod_{\sigma\in\text{Aut}(\f_{q^n}/\f_q)}G^\sigma$ for some $a\in\f_q^*$, where $\text{Aut}(\f_{q^n}/\f_q)$ is the Galois group of $\f_{q^n}$ over $\f_q$ and $G^\sigma$ is the polynomial obtained from $G$ by applying $\sigma$ to its coefficients. Since $V_{\f_q^2}(F)=\bigcup_{\sigma\in \text{Aut}(\f_{q^n}/\f_q)}V_{\f_q^2}(G^\sigma)$, where $V_{\f_q^2}(G^\sigma)$ is independent of $\sigma$, we have $V_{\f_q^2}(G^\sigma)=V_{\f_q^2}(F)$ for all $\sigma\in\text{Aut}(\f_{q^n}/\f_q)$. Hence
 $V_{\f_q^2}(F)=\bigcap_{\sigma\in \text{Aut}(\f_{q^n}/\f_q)}V_{\f_q^2}(G^\sigma)\subset\bigcap_{\sigma\in \text{Aut}(\f_{q^n}/\f_q)}V_{\overline\f_q^2}(G^\sigma)$. By B\'ezout's theorem, 
\[
\Bigl|\bigcap_{\sigma\in \text{Aut}(\f_{q^n}/\f_q)}V_{\overline\f_q^2}(G^\sigma)\Bigr|\le(\deg G)^2\le\frac 14(\deg F)^2.
\]

The Hasse-Weil bound and its variations have many applications in the study of polynomial equations over finite fields. In this paper, we use the above observations to prove the following result.

\begin{thm}\label{T1}
Assume that two rational functions $f(X),g(X)\in\f_q(X)\setminus\f_q$ with $\deg f=d$ and $\deg g=\delta$ satisfy the following conditions.
\begin{itemize}
\item[(i)] $f(\f_q)\subset g(\f_q^\dagger)$.
\item[(ii)] For each $a\in\f_q^\dagger$, with at most $8(d+\delta)$ exceptions, $\bigl|g|_{\f_q}^{-1}(g(a))\bigr|>\delta/2$.
\item[(iii)] $q\ge(d+\delta)^4$.
\end{itemize}
Then there exists $h(X)\in\f_q(X)$ such that $f(X)=g(h(X))$.
\end{thm}

Two special cases of Theorem~\ref{T1}, where $q$ is even and $g(X)=X^2+X$, and where $q$ is odd and $g(X)=X^2$, have appeared in a less explicit form in some recent studies on permutation polynomials \cite{Hou-CC, Hou-Tu-Zeng-ppt}. In fact, Theorem~\ref{T1} is motivated by these two special cases. A generalization of Theorem~\ref{T1} to multivariate rational functions is given in Section 4.

\section{Proof of Theorem~\ref{T1}}

Let $d'=d+\delta$. Write $f(X)=A(X)/B(X)$ and $g(X)=P(X)/Q(X)$, where $A,B,P,Q\in\f_q[X]$, $BQ\ne0$, and $\text{gcd}(A,B)=\text{gcd}(P,Q)=1$. Define
\begin{equation}\label{2.1}
F(X,Y)=A(X)Q(Y)-B(X)P(Y)\in\f_q[X,Y].
\end{equation}
If $F(X,Y)$, viewed as a polynomial in $Y$ over $\f_q(X)$, has a root $h(X)\in\f_q(X)$, then 
\[
A(X)Q(h(X))-B(X)P(h(X))=0,
\]
i.e., $f(X)=A(X)/B(X)=P(h(X))/Q(h(X))=g(h(X))$. (Of course, the converse of this statement is also true.) We will show that under conditions (i) -- (iii), $F(X,Y)\in\f_q(X)[Y]$ has a root in $\f_q(X)$.

We first show that (i) and (ii) imply a lower bound for $|V_{\f_q^2}(F)|$.
Let 
\[
\mathcal Y_1=\{a\in\f_q^\dagger:\bigl|g|_{\f_q}^{-1}(g(a))\bigr|>\delta/2\},\quad \mathcal Y_2=\{a\in\f_q^\dagger:\bigl|g|_{\f_q}^{-1}(g(a))\bigr|\le\delta/2\}.
\]
Note that $V_{\f_q^2}(F)=\{(x,y)\in\f_q\times\f_q:f(x)=g(y)\}$. By (i) and (ii),
\begin{align}\label{2.2}
|V_{\f_q^2}(F)|=\,&|\{(x,y)\in\f_q\times\f_q:f(x)=g(y)\}|\\
=\,&\sum_{x\in\f_q}\bigl|g|_{\f_q}^{-1}(f(x))\bigr|
\ge\sum_{x\in\f_q,\, f(x)\in g(\mathcal Y_1)}\bigl|g|_{\f_q}^{-1}(f(x))\bigr|\cr
\ge\,& \Bigl(\Bigl\lfloor\frac \delta 2\Bigr\rfloor+1\Bigr)|\{x\in\f_q:f(x)\in g(\mathcal Y_1)\}|\cr
\ge\,& \Bigl(\Bigl\lfloor\frac \delta 2\Bigr\rfloor+1\Bigr)(q-|\{x\in\f_q:f(x)\in g(\mathcal Y_2)\}|)\cr
\ge\,& \Bigl(\Bigl\lfloor\frac \delta 2\Bigr\rfloor+1\Bigr)(q-d\,|g(\mathcal Y_2)|)
\ge \Bigl(\Bigl\lfloor\frac \delta 2\Bigr\rfloor+1\Bigr)(q-d\,|\mathcal Y_2|)\cr
\ge\,& \Bigl(\Bigl\lfloor\frac \delta 2\Bigr\rfloor+1\Bigr)(q-8dd')
\ge q\Bigl(\Bigl\lfloor\frac \delta 2\Bigr\rfloor+1\Bigr)-8\delta dd'\cr
\ge\,&  q\Bigl(\Bigl\lfloor\frac \delta 2\Bigr\rfloor+1\Bigr)-2d'^3.
\nonumber
\end{align}

Write $F=p_1\cdots p_m$, where $p_i\in\f_q[X,Y]$ is irreducible with $\deg p_i=d_i$. We claim that $\deg_Yp_i>0$ for all $1\le i\le m$. Otherwise, for some $i$, $p_i(X,Y)=p_i(X)$. Since $P(Y)/Q(Y)$ is not constant, there exist $y_1,y_2\in\overline \f_q$ such that $P(y_1)/Q(y_1)\ne P(y_2)/Q(y_2)$, i.e.,
\[
\left|\begin{matrix} Q(y_1)& P(y_1)\cr Q(y_2) &P(y_2)\end{matrix}\right|\ne 0.
\]
Since $p_i(X)$ divides $F(X,y_1)$ and $F(X,y_2)$, where
\[
\left[\begin{matrix}F(X,y_1)\cr F(X,y_2)\end{matrix}\right]=\left[\begin{matrix} Q(y_1)& P(y_1)\cr Q(y_2) &P(y_2)\end{matrix}\right]\left[\begin{matrix}A(X)\cr B(X)\end{matrix}\right],
\]
we see that $p_i(X)$ divides both $A(X)$ and $B(X)$. This is impossible since $\text{gcd}(A,B)=1$. Hence the claim is proved.

If $\deg_Yp_i=1$ for some $i$, then $F(X,Y)$, as a polynomial in $Y$ over $\f_q(X)$, has a root in $\f_q(X)$, and we are done.

Now assume that $\deg_Yp_i\ge 2$ for all $1\le i\le m$. We will derive a contradiction to (iii). First, we claim that $\deg_YF=\delta$. Otherwise, from \eqref{2.1}, we see that $\deg Q=\deg P$ and $A(X)$ is a scalar multiple of $B(X)$. This is impossible since $\text{gcd}(A(X),B(X))=1$. Now we have
\[
\delta=\deg_YF=\sum_{i=1}^m\deg_Yp_i\ge 2m,
\]
so
\begin{equation}\label{2.3.1}
m\le\Bigl\lfloor\frac \delta 2\Bigr\rfloor.
\end{equation}
If $p_i$ is absolutely irreducible, by \eqref{1.1.1},
\begin{equation}\label{2.4}
|V_{\f_q^2}(p_i)|\le q+1+(d_i-1)(d_i-2)q^{1/2}.
\end{equation}
If $p_i$ is not absolutely irreducible, then by \eqref{1.3},
\begin{equation}\label{2.5}
|V_{\f_q^2}(p_i)|\le\frac 14 d_i^2.
\end{equation}
By \eqref{2.4} and \eqref{2.5}
\begin{equation}\label{2.6}
|V_{\f_q^2}(F)|\le\sum_{i=1}^m |V_{\f_q^2}(p_i)|\le\sum_{i=1}^m\bigl(q+1+(d_i-1)(d_i-2)q^{1/2}\bigr).
\end{equation}
Treating $d_1,\dots,d_m$ as real variables and using Lagrange multipliers, it is easy to see that under the condition $\sum_{i=1}^md_i=\deg F$, the quantity $\sum_{i=1}^m(d_i-1)(d_i-2)$ attains its maximum value when $d_1=\cdots=d_m=(\deg F)/m$. Thus \eqref{2.6}, combined with the fact that $\deg F\le d+\delta=d'$, gives
\begin{equation}\label{2.7}
|V_{\f_q^2}(F)|\le m(q+1)+q^{1/2}m\Bigl(\frac{d'}m-1\Bigr)\Bigl(\frac{d'}m-2\Bigr).
\end{equation}
By \eqref{2.2}, \eqref{2.3.1} and \eqref{2.7}, we have
\[
m(q+1)+q^{1/2}m\Bigl(\frac{d'}m-1\Bigr)\Bigl(\frac{d'}m-2\Bigr)\ge q(m+1)-2d'^3,
\]
i.e.,
\[
q-m\Bigl(\frac{d'}m-1\Bigr)\Bigl(\frac{d'}m-2\Bigr)q^{1/2}-2d'^3-m\le 0.
\]
Hence 
\[
q^{1/2}\le \frac12\Bigl[ m\Bigl(\frac{d'}m-1\Bigr)\Bigl(\frac{d'}m-2\Bigr)
+\Big( m^2\Bigl(\frac{d'}m-1\Bigr)^2\Bigl(\frac{d'}m-2\Bigr)^2+4(2d'^3+m)\Bigr)^{1/2}\Bigr].
\]
In the above, 
\[
m\Bigl(\frac{d'}m-1\Bigr)\Bigl(\frac{d'}m-2\Bigr)\le (d'-1)(d'-2)
\]
and
\begin{align*}
&m^2\Bigl(\frac{d'}m-1\Bigr)^2\Bigl(\frac{d'}m-2\Bigr)^2+4(2d'^3+m)\cr
&< (d'-1)^2(d'-2)^2+8d'^3+2d' \cr
&<(2d'^2-(d'-1)(d'-2))^2.
\end{align*}
Hence 
\[
q^{1/2}<\frac12\bigl[(d'-1)(d'-2)+2d'^2-(d'-1)(d'-2)\bigr]=d'^2,
\]
which is a contradiction to (iii).

\section{Remarks}\label{s3}

In the proof of Theorem~\ref{T1}, conditions (i) and (ii) were only used to derive \eqref{2.2}. In fact, a modification of the conditions in Theorem~\ref{T1} gives the following more convenient form of the theorem.

\begin{thm}\label{T3.1}
Let $f(X),g(X)\in\f_q(X)\setminus\f_q$ be such that $\deg f=d$ and $\deg g=\delta$. If there is a constant $0<\epsilon\le 1$ such that
\begin{equation}\label{3.1}
|\{(x,y)\in\f_q\times\f_q:f(x)=g(y)\}|\ge q\Bigl(\Bigl\lfloor\frac\delta 2\Bigr\rfloor+\epsilon\Bigr),
\end{equation}
and $q\ge(d+\delta)^4/\epsilon^2$, then $f=g\circ h$ for some $h\in\f_q(X)$. 
\end{thm}

The proof of Theorem~\ref{T3.1} is almost identical to that of Theorem~\ref{T1}. One replaces \eqref{2.2} with \eqref{3.1} and proceeds accordingly. We leave the details to the reader.

If $f(X),g(X)\in\f_q(X)\setminus\f_q$ are such that
\begin{equation}\label{3.2}
\Bigl|\Bigl\{a\in\f_q:\bigl|g|_{\f_q}^{-1}(g(a))\bigr|\le\frac{\deg g}2\Bigr\}\Bigr|=o(q).
\end{equation}
and 
\begin{equation}\label{3.3}
|\{x\in\f_q:f(x)\notin g(\f_q)\}|=o(q),
\end{equation}
then \eqref{3.1} is satisfied for a suitable $\epsilon>0$ when $q$ is sufficiently large. If $f=g\circ h$ for some $h\in \f_q(X)$, then certainly $f(\f_q^\dagger)\subset g(\f_q^\dagger)$. The reader might wonder why not simply state condition \eqref{3.3} as $f(\f_q^\dagger)\subset g(\f_q^\dagger)$. The reason is that in applications that we anticipate, the latter may not be as easy to prove as the former.

To see that \eqref{3.2} and \eqref{3.3} imply \eqref{3.1}, let $d=\deg f$, $\delta=\deg g$, 
\begin{align*}
\mathcal X\,&=\{x\in\f_q:f(x)\in g(\f_q)\},\cr
\mathcal Y_1&=\{a\in\f_q:\bigl|g|_{\f_q}^{-1}(g(a))\bigr|>\delta/2\},\cr
\mathcal Y_2&=\{a\in\f_q:\bigl|g|_{\f_q}^{-1}(g(a))\bigr|\le\delta/2\}.
\end{align*}
Then 
\begin{align*}
|\{(x,y)\in\f_q\times\f_q:f(x)=g(y)\}|\,
&=\sum_{x\in\f_q}|g^{-1}(f(x))|\ge\sum_{x\in\mathcal X,\,f(x)\in g(\mathcal Y_1)}|g^{-1}(f(x))|\cr
&\ge\Bigl(\left\lfloor\frac \delta 2\right\rfloor+1\Bigr)|\{x\in\mathcal X:f(x)\in g(\mathcal Y_1)\}|.
\end{align*}
In the above
\begin{align*}
|\{x\in\mathcal X:f(x)\in g(\mathcal Y_1)\}|\,&=|\mathcal X|-|\{x\in\mathcal X:f(x)\in g(\mathcal Y_2)\}|\cr
&\ge|\mathcal X|-d\,|g(\mathcal Y_2)|\ge|\mathcal X|-d\,|\mathcal Y_2|=q-o(q).
\end{align*}
Then
\[
|\{(x,y)\in\f_q\times\f_q:f(x)=g(y)\}|\ge\Bigl(\left\lfloor\frac \delta 2\right\rfloor+1\Bigr)(q-o(q)),
\]
and hence \eqref{3.1} is satisfied for some $\epsilon>0$ when $q$ is sufficiently large. 

Let $\mathcal V=\{v\in g(\f_q):\bigl|g|_{\f_q}^{-1}(v)\bigr|\le\delta/2\}$. Then $\mathcal V=g(\mathcal Y_2)$. Since $|g(\mathcal Y_2)|\le|\mathcal Y_2|\le(\delta/2)|g(\mathcal Y_2)|$, we see that $|\mathcal Y_2|=o(q)$ if and only if $|g(\mathcal Y_2)|=o(q)$, that is, \eqref{3.2} holds if and only if $|\mathcal V|=o(q)$. Therefore, rational functions $g\in\f_q(X)$ satisfying \eqref{3.2} are those such that each of the values of $g$ on $\f_q$, with the exception of an $o(q)$ number of them, is attained more that $(\deg g)/2$ times. There are examples of rational functions satisfying \eqref{3.2}.

\begin{exmp}\label{E3.1}\rm
Let $\f_r\subset\f_q$ and let $g(X)=X^r-X$. In this case, $g$ induces an $\f_r$-map $\f_q\to\f_q$ whose kernel is $\f_r$.
\end{exmp}

\begin{exmp}\label{E3.1a}\rm
Example~\ref{E3.1} can be made more general. Let $p=\text{char}\,\f_q$, and for any $\f_p$-subspace $U$ of $\f_q$, let $g(X)=\prod_{u\in U}(X-u)$. In this case, $g$ induces an $\f_p$-map $\f_q\to\f_q$ whose kernel is $U$ (\cite[Theorem~3.52]{Lidl-Niederreiter-FF-1997}). 
\end{exmp}

\begin{exmp}\label{E3.2}\rm
Let $d\mid q-1$ and let $g(X)=X^d$. In this case, $g$ induces a group homomorphism $\f_q^*\to\f_q^*$ whose kernel is of size $d$.
\end{exmp}

Moreover, if $g\in\f_q(X)$ satisfies \eqref{3.2}, then so do $g\circ\phi$ and $\phi\circ g$ for any $\phi\in\f_q(X)$ with $\deg \phi=1$. Are there other examples? What more can we say about rational functions satisfying \eqref{3.2}? These appear to be interesting questions in their own right.

\section{Generalization to Multivariate Rational Functions}

The Hasse-Weil bound has a generalization for absolutely irreducible polynomials in $n$ variables over $\f_q$, which is the Lang-Weil bound \cite{Cafure-Matera-FFA-2006, Lang-Weil-AJM-1954}:

\begin{thm}[Lang and Weil \cite{Lang-Weil-AJM-1954}]\label{T-LW}
Let $F(X_1,\dots,X_n)\in\f_q[X_1,\dots,X_n]$ be absolutely irreducible of degree $d$. Then 
\[
\bigl||V_{\f_q^n}(F)|-q^{n-1}\bigr|\le (d-1)(d-2)q^{n-3/2}+c(n,d)q^{n-2},
\]
where $c(n,d)$ is a constant depending only on $n$ and $d$.
\end{thm}

Cafure and Matera \cite{Cafure-Matera-FFA-2006} provided an explicit expression for the constant $c(n,d)$:

\begin{thm}[Cafure and Matera \cite{Cafure-Matera-FFA-2006}]\label{T-CM}
Let $F(X_1,\dots,X_n)\in\f_q[X_1,\dots,X_n]$ be absolutely irreducible of degree $d$. Then 
\[
\bigl||V_{\f_q^n}(F)|-q^{n-1}\bigr|\le (d-1)(d-2)q^{n-3/2}+5d^{13/3}q^{n-2}.
\]
\end{thm}

In this section, we generalize Theorem~\ref{3.1} to multivariate rational functions using Theorem~\ref{T-CM}. For convenience, we write $\X=(X_1,\dots,X_n)$. For $f(\X)\in\f_q(\X)$ written in the form $f(\X)=A(\X)/B(\X)$, where $\text{gcd}(A,B)=1$, and for $\x\in\f_q^n$, we say that $f$ is defined at $\x$ if $A(\x)$ and $B(\x)$ are not both $0$; in this case, $f(\x)\in\f_q\cup\{\infty\}$.

\begin{thm}\label{T4.1}
Let $f(\X)\in\f_q(\X)\setminus\f_q$ and $g(X)\in\f_q(X)\setminus\f_q$ be such that $\deg f=d$ and $\deg g=\delta$. If there is a constant $0<\epsilon\le 1$ such that
\begin{equation}\label{4.1}
|\{(\x,y)\in\f_q^n\times\f_q:\text{$f$ is defined at $\x$ and $f(\x)=g(y)$}\}|\ge q^n\Bigl(\Bigl\lfloor\frac\delta 2\Bigr\rfloor+\epsilon\Bigr),
\end{equation}
and $q\ge7.8(d+\delta)^{13/3}/\epsilon^2$, then $f=g\circ h$ for some $h\in\f_q(\X)$. 
\end{thm} 

\begin{proof}
Write $f(\X)=A(\X)/B(\X)$ and $g(X)=P(X)/Q(X)$, where $A(\X),B(\X)\in\f_q[\X]$, $B\ne 0$, $\text{gcd}(A,B)=1$, and $P(X),Q(X)\in\f_q[X]$, $Q\ne 0$, $\text{gcd}(P,Q)=1$. Let 
\[
F(\X,Y)=A(\X)Q(Y)-B(\X)P(Y)\in\f_q[\X,Y].
\]
By \eqref{4.1},
\begin{equation}\label{4.2}
|V_{\f_q^{n+1}}(F)|\ge q^n\Bigl(\Bigl\lfloor\frac\delta 2\Bigr\rfloor+\epsilon\Bigr).
\end{equation}
Write $F=p_1,\cdots p_m$, where $p_i\in\f_q[\X,Y]$ is irreducible with $\deg p_i=d_i$.

We first claim that $\deg_Yp_i>0$ for all $1\le i\le n$. The proof of this claim is identical to that of the univariate case; see the paragraph in Section 2 after \eqref{2.2}.

If $\deg_Yp_i=1$ for some $i$, say $p_i=\alpha(\X)Y+\beta(\X)$, where $\alpha(\X),\beta(\X)\in\f_q(\X)$ and $\alpha(\X)\ne 0$. Let $h(\X)=-\beta(\X)/\alpha(\X)$. Then 
\[
0=F(\X,h(\X))=A(\X)Q(h(\X))-B(\X)P(h(\X)),
\]
i.e., $f(\X)=g(h(\X))$, and we are done.

Now assume that $\deg_Yp_i\ge 2$ for all $1\le i\le m$. We will derive a contradiction. If $p_i$ is absolutely irreducible, by Theorem~\ref{T-CM},
\[
|V_{\f_q^{n+1}}(p_i)|\le q^n+(d_i-1)(d_i-2)q^{n-1/2}+5d_i^{13/3}q^{n-1}.
\]
If $p_i$ is not absolutely irreducible, by \cite[Lemma~2.3]{Cafure-Matera-FFA-2006},
\[
|V_{\f_q^{n+1}}(p_i)|\le\frac 14d_i^2q^{n-1}.
\]
Hence
\begin{equation}\label{4.3}
|V_{\f_q^{n+1}}(F)|\le\sum_{i=1}^m|V_{\f_q^{n+1}}(p_i)|\le\sum_{i=1}^m\bigl(q^n+(d_i-1)(d_i-2)q^{n-1/2}+5d_i^{13/3}q^{n-1}\bigr).
\end{equation}
Treating $d_1,\dots,d_m$ as real variables and using Lagrange multipliers, we see that under the condition $\sum_{i=1}^md_i=\deg F$, both $\sum_{i=1}^m(d_i-1)(d_i-2)$ and $\sum_{i=1}^md_i^{13/3}$ attain their maximum values when $d_1=\cdots=d_m=(\deg F)/m$. Let $d'=d+\delta$ and note that $\deg F\le d'$. Now by \eqref{4.3},
\begin{equation}\label{4.4}
|V_{\f_q^{n+1}}(F)|\le mq^n+m\Bigl(\frac{d'}m-1\Bigr)\Bigl(\frac{d'}m-2\Bigr)q^{n-1/2}+5m\Bigl(\frac{d'}m\Bigr)^{13/3}q^{n-1}.
\end{equation}
As seen in Section 2, we have $\deg_YF=\delta$. Hence
\[
\delta=\deg_YF=\sum_{i=1}^m\deg_Yp_i\ge 2m,
\]
so
\[
m\le\left\lfloor\frac\delta 2\right\rfloor.
\]
Thus by \eqref{4.2},
\begin{equation}\label{4.5}
|V_{\f_q^{n+1}}(F)|\ge q^n(m+\epsilon).
\end{equation}
Combining \eqref{4.4} and \eqref{4.5} gives
\[
q^n(m+\epsilon)\le mq^n+m\Bigl(\frac{d'}m-1\Bigr)\Bigl(\frac{d'}m-2\Bigr)q^{n-1/2}+5m^{-10/3}{d'}^{13/3}q^{n-1},
\]
i.e.,
\[
\epsilon q-m\Bigl(\frac{d'}m-1\Bigr)\Bigl(\frac{d'}m-2\Bigr)q^{1/2}-5m^{-10/3}{d'}^{13/3}\le 0.
\]
Hence
\begin{align*}
&q^{1/2}\cr
&\le\frac 1{2\epsilon}\Bigl[m\Bigl(\frac{d'}m-1\Bigr)\Bigl(\frac{d'}m-2\Bigr)+\Bigl(m^2\Bigl(\frac{d'}m-1\Bigr)^2\Bigl(\frac{d'}m-2\Bigr)^2+20\epsilon m^{-10/3}{d'}^{13/3}\Bigr)^{1/2}\Bigr]\cr
&<\frac 1{2\epsilon}\bigl({d'}^2+({d'}^4+20{d'}^{13/3})^{1/2}\bigr)<\frac 1{2\epsilon}\bigl({d'}^2+\sqrt{21}{d'}^{13/6}\bigr)<\frac {1+\sqrt{21}}{2\epsilon}{d'}^{13/6}.
\end{align*}
Hence 
\[
q<\Bigl(\frac {1+\sqrt{21}}2\Bigr)^2\,\frac{{d'}^{13/3}}{\epsilon^2}<7.8\,\frac{{d'}^{13/3}}{\epsilon^2},
\]
which is a contradiction.
\end{proof}

If $f(\X)\in\f_q(\X)\setminus\f_q$ and $g(X)\in\f_q(X)\setminus\f_q$ are such that 
\[
\Bigl|\Bigl\{a\in\f_q:\bigl|g|_{\f_q}^{-1}(g(a))\bigr|\le\frac{\deg g}2\Bigr\}\Bigr|=o(q)
\]
and
\[
|\{\boldsymbol{x}\in\f_q^n:f(\boldsymbol{x})\notin g(\f_q)\}|=o(q^n),
\]
then \eqref{4.1} is satisfied for a suitable $\epsilon>0$ when $q$ is sufficiently large. The proof of this claim is identical to that of the univariate case given in Section~\ref{s3}.



\end{document}